\documentclass[a4paper,12pt]{article} 

\title{Extended cyclic codes, maximal arcs and ovoids}
\author{Kanat Abdukhalikov and Duy Ho\\	
UAE University, PO Box 15551, Al Ain, UAE\\
abdukhalik@uaeu.ac.ae, duyho92@gmail.com}

\date{ } 

\usepackage{amsthm} 
\usepackage{amsmath,amssymb} 

\begin{document} 

\maketitle

\theoremstyle{plain} 
\newtheorem{theorem}{Theorem}
\newtheorem{corollary}{Corollary}
\newtheorem{lemma}{Lemma}
\newtheorem{observation}{Observation}
\newtheorem{proposition}{Proposition}

\theoremstyle{definition} 
\newtheorem{definition}{Definition}
\newtheorem{claim}{Claim}
\newtheorem{fact}[theorem]{Fact}
\newtheorem{assumption}{Assumption}
\newtheorem{remark}{Remark}
\newtheorem{example}{Example} 

\begin{abstract}
We  show that extended cyclic codes over $\mathbb{F}_q$ with parameters $[q+2,3,q]$, $q=2^m$, 
determine regular hyperovals. 
We also show that  extended cyclic  codes with parameters $[qt-q+t,3,qt-q]$, $1<t<q$, $q$ is a power of $t$, 
determine (cyclic) Denniston maximal arcs. Similarly, cyclic codes with parameters $[q^2+1,4,q^2-q]$    
are equivalent to ovoid codes obtained from elliptic quadrics in $PG(3,q)$. 
Finally, we give simple presentations of Denniston maximal arcs in $PG(2,q)$ and elliptic quadrics in $PG(3,q)$.  
\end{abstract} 


\section{Introduction} 

In \cite[Chapter 12]{Ding} Ding described a construction of extended cyclic codes $\mathbb{F}_q$ with parameters 
$[q+2,3,q]$, $q=2^m$. 
Codes with these parameters are MDS codes and they determine hyperovals in the projective plane $PG(2,q)$. 
Ding raised a question to find out which hyperovals are obtained by such a construction. 
Furthermore, he also posed a question  which hyperovals can be obtained from arbitrary extended cyclic codes 
with parameters  $[q+2,3,q]$. In the first part of the paper, we answer these questions in Theorem \ref{ext-MDS}. 

In our investigation we use the polar presentation \cite{Ab2017,Ab2019}  of the projective plane $PG(2,q)$. 
With this presentation,  points of a hyperoval can be chosen to be in an affine plane  $AG(2,q)$.  
As a result,  the connection between mentioned cyclic linear codes and hyperovals are described in a natural way and 
related reasonings on properties of  linear codes are greatly simplified. 
 
Generalizing the method applied in the proof of Theorem \ref{ext-MDS}, further we study linear codes $C$ with 
parameters $[qt-q+t,3,qt-q]$, $1<t<q$. 
It is known that if such $C$ is a projective code or a two-weight code, 
then it uniquely determines a maximal arc in $PG(2,q)$ \cite{Calderbank}.  
Moreover, De Winter et al. \cite{Ding2019} constructed a maximal arc $\mathcal K$ that is invariant under a cyclic group 
that fixes one point of $\mathcal K$ and acts regularly on the set of the remaining points of $\mathcal K$ 
(in their construction it is assumed that $t-1$ divides $q-1$ and $t$ divides $q$, or equivalently, $q$ is a power of $t$). 
This maximal arc belongs to the class  of arcs constructed by Denniston \cite{Denniston}. 
Based on the properties of $\mathcal{K}$, De Winter et al. \cite{Ding2019} constructed an extended cyclic  code 
with parameters $[qt-q+t,3,qt-q]$ and nonzero weights $qt-q$ and $qt-q+t$. 

In the converse direction of \cite{Ding2019}, we show that if $C$ is an extended cyclic code with parameters 
$[qt-q+t,3,qt-q]$, $q$ is a power of $t$, then it determines a (cyclic) Denniston maximal arc.  
This result, together with the construction in \cite{Ding2019}, fully characterizes extended cyclic codes with 
parameters $[qt-q+t,3,qt-q]$, where $q$ is a power of $t$. 
We also show that in terms of polar coordinates  the Denniston maximal arcs have very simple presentations, 
and in our particular case the corresponding Denniston maximal arcs can be expressed in a form  
that clearly shows their cyclicity. 

Ding \cite{Ding} introduced a family of cyclic codes with parameters $[q^2+1,4,q^2-q]$ and stated 
without proof that they can be  obtained from elliptic quadrics in $PG(3,q)$. 
We provide a simple proof of this statement. We also introduce a coordinate-free presentation of elliptic 
quadrics in $PG(3,q)$. 

The content of the paper is organized as follows. 
In Section \ref{sec-prelim}, we recall some results on finite geometry and projective linear codes. 
In Section \ref{sec-hyperovals},  we consider extended cyclic codes with parameters  $[q+2,3,q]$  and  their 
connections with hyperovals. 
Likewise, in Section \ref{sec-arcs},  we consider   maximal arcs and extended cyclic codes 
with parameters  $[qt-q+t,3,qt-q]$. 
Further in Section \ref{sec-ovoids} we consider cyclic codes with parameters $[q^2+1,4,q^2-q]$ and their connections 
with elliptic quadrics in $PG(3,q)$. 
\section{Preliminaries}  
\label{sec-prelim} 

Let $F=\mathbb{F}_{2^m}$ be a finite field of order $q=2^m$. Consider $F$ as a subfield of $K=\mathbb{F}_{2^n}$, 
where $n=2m$, so $K$ is a two dimensional vector space over $F$. 

Let $Tr_{M/L}(x)$ and $N_{M/L}(x)$ be the trace and the norm functions with respect to a finite field extension $M/L$. 
The \emph{conjugate} of $x\in K$ over $F$ is
$$\bar{x} =x^q.$$
Then the \emph{trace}  and the \emph{norm} maps from $K$ to $F$ are 
$$T(x) = Tr_{K/F}(x)= x + \bar{x}= x + x^q,$$ 
$$N(x) = N_{K/F} (x) = x \bar{x} = x^{1+q}.$$
The \emph{unit circle} of $K$ is the set of elements of norm $1$:
$$S= \{ u\in K \mid N(u) =1 \}= \{ u\in K \mid  u^{q+1} =1 \}. $$
Therefore, $S$ is the multiplicative group of $(q+1)$th roots of unity in $K$. Note that $\bar{u} =u^{-1}$ for $u\in S$. 
Since $F\cap S = \{ 1\}$, each non-zero element of $K$ has  a unique polar coordinate representation
 $$ x=\lambda u$$
 with $\lambda\in F^*$ and $u \in S$. For any $x\in K^*$ we have $\lambda = \sqrt{x \bar{x}}$, $u= \sqrt{x/ \bar{x}}$. 

Consider points of the projective plane $PG(2,q)$ in homogeneous coordinates as triples $(x ,y,z)$, 
where $x, y, z \in F$, $(x,y,z)\neq (0,0,0)$, 
and we identify $(x,y,z)$ with $(\lambda x , \lambda y , \lambda z)$, $\lambda\in F^*$. 
We shall call points of the form $(x,y,0)$  the points at infinity. 
So $z=0$ indicates the line at infinity (it consists of all points at infinity). 
We define an affine plane $AG(2,q)$ which is obtained by removing all points at infinity from $PG(2,q)$, 
so points of this affine plane $AG(2,q)$ are 
$$\{ (x,y,1) \mid x, y \in F \} .$$
Associating $(x,y,1)$ with $(x,y)$ we can identify points of the affine plane $AG(2,q)$ with  elements 
of the vector space  
$$V(2,q) = \{ (x,y) \mid x, y \in F \} .$$

We introduce now another representation of $PG(2,q)$ using the field $K$. 
Consider pairs $(x,z)$, where $x \in K$, $z\in F$, $x\neq 0$ or 
$z\neq 0$, and we identify $(x,z)$ with $(\lambda x ,\lambda z)$, $\lambda\in F^*$. 
Then points of $PG(2,q)$ are 
$$\{ (x ,1) \mid \ x \in K \} \cup \{ (u,0) \mid u\in S\}.$$
The elements from $\{ (u,0) \mid u\in S\}$,  will be referred to as the points at infinity. 
An affine plane $AG(2,q)$ is obtained by removing all points at infinity from $PG(2,q)$, so the points of this 
affine plane $AG(2,q)$ are 
$$\{ (x , 1) \mid x \in K \} .$$
Associating $(x , 1)$ with $x \in K$ we can identify the points of the affine plane $AG(2,q)$ with the elements of 
the field $K$. 

Let $\xi \in K\setminus F$. Then $\{ 1, \xi\}$ is a basis of $K$ over F.  If for $w\in K$ we have $w =x+y\xi$, $x, y \in F$, 
then point $(w,1)$ can be written as $(x,y,1)$ in traditional coordinates.

An \emph{oval} in $PG(2,q)$ is a set of $q+1$ points, no three of which are collinear.  
Any line of the plane meets the oval $\mathcal{O}$ at either
0, 1 or 2 points and is called an exterior, tangent or secant, respectively. 
All the tangent lines to the oval $\mathcal{O}$ concur  \cite{Hir} at the same point $N$, called the \emph{nucleus} 
 of $\mathcal{O}$. The set $\mathcal{H}=\mathcal{O} \cup \{ N\}$ becomes a \emph{hyperoval}, that is a set of $q+2$
points, no three of which are collinear. 
Conversely, by removing any point from a hyperoval one gets an oval. 

The set $S\cup \{ 0\}$ is a regular hyperoval (hyperconic) in $K$, considered as an affine plane over $F$ 
(see \cite{Ab2019}).   

Let $C$ be a linear $[n,3]$-code with a generator matrix $G$. Let $G$ contain no zero columns. Consider 
columns $P_j$ of the matrix $G$ as generators of points in $PG(2,q)$. Then the multi-set 
$\mathcal{P} = \{\{P_1, P_2, \dots, P_n \}\} $ contains $n$ points from $PG(2,q)$ (some of them may be equal). 
The following lemma will be very useful in our investigations (\cite[Theorem 2.36]{Ding} or  \cite[Lemma 4.15]{Ball}).

\begin{lemma} 
\label{main} 
The multi-set $\mathcal{P}$ of columns of a generator matrix $G$ of a $[n,3,n-t]$-code over $F$ is a multi-set 
of points of $PG(2,q)$ in which every line of $PG(2,q)$ contains at most $t$ points of $\mathcal{P}$, and some 
line of $PG(2,q)$ contains exactly $t$ points of $\mathcal{P}$. 
\end{lemma} 

A linear code $[n,3]$-code $C$ is called {\it projective} if the columns of its generator matrix generate different points in 
$PG(2,q)$. It means that the dual code $C^\perp$ does not contain words of weight 2. 
The assumption that no column of $G$ is 0 means the absence of weight 1 words in $C^\perp$. 
A code $C$ is projective if and only if its dual code $C^\perp$ has minimum distance $d^\perp \ge 3$. 

We say that two codes are equivalent if one can be obtained from the other by a permutation of the coordinates. 


\section{Extended cyclic codes and hyperovals}
\label{sec-hyperovals}

Let $S=\langle \beta \rangle$. 
Let $C$ be a cyclic code over $K$ of length $q+1$ with the generator polynomial $x-\beta$. 
Then $C$ is a $[q+1,q,2]$ MDS code over $K$.  
The subfield subcode $C_F$ over $F$ is generated by the polynomial 
$(x-\beta)(x-\bar{\beta}) = x^2 +T(\beta)x +1$, so $\beta$ and $\bar{\beta}$ are zeroes of the cyclic code $C_F$. 
Define $1\times (q+1)$ matrix  
$$
L=\left(
\begin{array}{ccccc}
1 & \beta & \beta^2 & \dots & \beta^q  
 \end{array}
\right).  
$$
Then by \cite[Theorem 4.4.3]{Huffman} we have 
$$ C_F = \{ c \in F^{q+1} \mid L c^T =0 \}. $$ 

Let $\xi \in K\setminus F$. Then $\{ 1, \xi\}$ is a basis of $K$ over F and let $\beta^i =x_i+y_i\xi$, $x_i, y_i \in F$.

Theorems 12.25-12.26 in \cite{Ding} presented a sequence of codes, we reproduce this sequence using 
polar coordinate representation of hyperovals. 

1) Theorem 12.25 in \cite{Ding} states that  $C_F$ is a $[q+1,q-1,3]$ MDS code. Its  parity-check matrix is  
$$
\left(
\begin{array}{cccc}
 x_0 & x_1 & \dots & x_q  \\
 y_0 & y_1 & \dots & y_q  
 \end{array}
\right); 
$$

2) Theorem 12.26 in \cite{Ding} states that the extended code $\overline{C_F}$ is a $[q+2,q-1,4]$ MDS code.  
Its parity-check matrix is  
$$
A= \left(
\begin{array}{ccccc}
 x_0 & x_1 & \dots & x_q  & 0 \\
 y_0 & y_1 & \dots & y_q  & 0 \\
  1 & 1 & \dots & 1  & 1
\end{array}
\right), 
$$
where the set 
$$\{ (x_0,y_0,1), (x_1,y_1,1), \dots, (x_q,y_q,1), (0,0,1)\}$$ 
is a regular hyperoval \cite{Ab2019}, corresponding to $S\cup \{ 0\}$;   

3) Theorem 12.26 in \cite{Ding} also states that the dual code $(\overline{C_F})^\perp$ is a $[q+2,3,q]$ MDS code.  
Clearly its generator matrix is $A$. 
  
In Problem 12.1 in \cite{Ding} it was posed as a question whether the hyperoval defined by the code 
$(\overline{C_F})^\perp$ is equivalent to a known hyperoval. 
We see that  the mentioned hyperoval is equivalent to the regular hyperoval $S\cup \{ 0\}$ (hyperconic). 

In Problem 12.3 in \cite{Ding} it was asked to characterize all extended cyclic codes over $F$ with parameters 
$[q+2,3,q]$. The  next theorem shows that all these codes are equivalent to  MDS codes obtained from 
the regular hyperovals. 

\begin{theorem} 
\label{ext-MDS} 
An extended cyclic code over $\mathbb{F}_q$ with parameters $[q+2,3,q]$ is equivalent to 
an MDS code obtained from a regular hyperoval. 
\end{theorem} 

\begin{proof} 
Let $C$ be an extended cyclic $[q+2,3,q]$-code. The parameters of the code show that $C$ is a MDS code. 
Then $C^\perp$ is  a MDS $[q+2,q-1,4]$-code \cite{Huffman}. 
Let $C'$ denote the code obtained by puncturing the extended coordinate in $C$. 
Then $C'$ is the original cyclic $[q+1,3]$ code. 
We will show that $C'$ has parity-check polynomial  $(x-\gamma)(x-\bar{\gamma})(x-1)$ for some $\gamma \in S$, 
$\gamma \ne 1$. 
Indeed, the factorization of $x^{q+1}-1$ over $K$ is given by 
$$x^{q+1}-1 = \prod_{u\in S} (x-u)= \prod_{i=0}^q (x-\beta^i),$$
where $\beta \in K$ is a primitive $(q+1)$-th root of unity and $S=\langle \beta \rangle$. 
Note that $(x-\beta^i)(x-\beta^{-i})=x^2-T(\beta^i)x+1\in F[x]$ is irreducible over $F$ for all $i$ with $1\le i\le q/2$, 
as $\beta^i \in S$. 
Then the corresponding factorization of $x^{q+1}-1$ over $F$ is given by 
$$x^{q+1}-1 = (x-1) \prod_{i=1}^{q/2} (x^2-T(\beta^i)x+1).$$
Consequently, every polynomial of degree 3 over $F$ dividing $x^{q+1}-1 $ must be of the form 
$(x-\gamma)(x-\bar{\gamma})(x-1)$ with $\gamma \in S\setminus \{1\}$.

Therefore, $(C')^\perp$ is a cyclic $[q+1,q-2]$-code with generator polynomial $(x-\gamma)(x-\bar{\gamma})(x-1)$. 
Define the following $2\times (q+1)$ matrix  over $K$: 
$$
L=\left(
\begin{array}{ccccc}
1 & \gamma & \gamma^2 & \dots & \gamma^q  \\
1 & 1 & 1 & \dots & 1
 \end{array}
\right).  
$$
Then by \cite[Theorem 4.4.3]{Huffman} we have 
$$ (C')^\perp = \{ c \in F^{q+1} \mid L c^T =0 \}. $$ 

 Let $\xi \in K \setminus F$ and $\gamma^i =x_i+y_i\xi$, where $x_i$ and  $y_i$ are from $F$. 
 Then $(C')^\perp$ is a cyclic code with a parity-check matrix 
$$
B=\left(
\begin{array}{cccc}
 x_0 & x_1 & \dots & x_q   \\
 y_0 & y_1 & \dots & y_q   \\
  1 & 1 & \dots & 1  
\end{array}
\right). 
$$
This means that  $C'$ is a cyclic code over $F$ with generator matrix $B$. Since 
$$0= \sum_{i=0}^q \gamma^i = \sum_{i=0}^q x_i + \xi\sum_{i=0}^q y_i,$$ 
we have 
$$\sum_{i=0}^q x_i = \sum_{i=0}^q y_i =0.$$
 
Since $C$ is an extended code of $C'$,  $C$ has the following generator matrix 
$$
A=\left(
\begin{array}{ccccc}
 x_0 & x_1 & \dots & x_q  & 0 \\
 y_0 & y_1 & \dots & y_q  & 0 \\
  1 & 1 & \dots & 1  & 1
\end{array}
\right). 
$$     
Since $C^\perp$ is  an MDS $[q+2,q-1,4]$-code  with parity-check matrix $A$, all elements $\gamma^i$, $0\le i\le q$, 
must be distinct (otherwise two columns of $A$ will be identical, and $C^\perp$ would have a codeword of weight 2), 
so $S=\langle \gamma \rangle$. Hence the columns of $A$ define a regular hyperoval.  \qed 
\end{proof}

Moreover, in Problem 12.2 in \cite{Ding} it was posed as a question whether every MDS code over $F$ with parameters 
$[q+2,3,q]$ is an extended cyclic code. We see that the answer to this question is negative, since there are 
hyperovals not equivalent to the regular hyperovals. 

\section{Extended cyclic codes and maximal arcs} 
\label{sec-arcs} 

In this section we show that results from Section \ref{sec-hyperovals} can be extended to maximal arcs. 

A $\{k;t\}$-arc in $PG(2,q)$ is a set $\mathcal K$ of $k$ points such that $t$ is the maximum number of points in 
$\mathcal K$ that are collinear. It is known that $k\le (q+1)(t-1)+1$. 
A $\{k;t\}$-arc in $PG(2,q)$ with $k = (q+1)(t-1)+1$ is called a \textit{maximal arc}. 
If $\mathcal K$ is a maximal $\{k;t\}$-arc in $PG(2,q)$ and $1<t < q$ then $q$ is even, $t$ is a divisor of $q$, 
and every line in $PG(2,q)$ intersects $\mathcal K$ in 0 or $t$ points. 
The $\{q+2;2\}$-arcs in $PG(2,q)$ are hyperovals. 

We recall a construction of maximal arcs proposed by Denniston \cite{Denniston}. 
Choose $\delta \in F$ such that the polynomial $X^2 + \delta X + 1$ is irreducible over $F$. 
For each $\lambda \in F$ consider the quadratic curve $D_{\lambda}$  in $AG(2,q)$ defined by the  equation 
$X^2 + \delta XY + Y^2 =\lambda$. If $\lambda \ne 0$ then $D{_\lambda}$ is a conic and its nucleus is the point $(0,0)$. 
If  $\lambda = 0$ then $D_{\lambda}$ consists of the single point $(0,0)$. Let $\Delta \subseteq F$. 
Then the set 
\begin{equation} 
\label{Denniston} 
D = \bigcup_{\lambda \in \Delta} D_\lambda 
\end{equation} 
is a maximal arc in $AG(2,q)$ (and therefore in $PG(2,q)$) if and only if $\Delta$ is a subgroup  of the additive group  
of $F$ \cite{Abatangelo,Denniston}.  In this case $D$ is a maximal $\{qt-q+t;t\}$-arc with $t=|\Delta |$. 

The next theorem shows that in terms of polar coordinates the Denniston maximal arcs can be expressed in a very simple way. 

\begin{theorem} 
The Denniston maximal arcs $(\ref{Denniston})$ can be expressed as 
\begin{equation} 
\label{Denniston2} D = \bigcup_{\lambda \in \Lambda} \lambda S \subset K,
\end{equation} 
where $\Lambda$ is a subgroup of the additive group  of the  field $F$ and $S$ is the unit circle of $K$. 
\end{theorem} 

\begin{proof} 
Let the polynomial  $X^2 + \delta X + 1$ be irreducible over $F$. It has two distinct conjugate roots 
$u$ and $\bar{u}$ in $K$, where $u\bar{u}=1$ and $u+\bar{u}=\delta$. 
Therefore, $u^{q+1}=u\bar{u}=1$ and $u \in S$, $u\ne 1$. Any element $z\in K$ can be written as $z=x+yu$, 
where $x, y \in F$. Then 
$$N(z) = N(x+yu) = (x+yu)(x+y\bar{u}) = x^2 + (u+\bar{u})xy + y^2 = x^2 + \delta xy + y^2.$$
Then  the set $\lambda S = \{ z\in K \mid N(z) = \lambda^2 \}$ corresponds to the set $D_{{\lambda}^2}$  and 
$\Delta = \{ \lambda^2 \mid \lambda \in \Lambda \}$.  \qed 
\end{proof} 

Clearly $D$ is a union of $|\Lambda|-1$ circles $\lambda S$ and the element $0$. 

Let $t-1$ divide $q-1$. De Winter et al. \cite[Theorem 1]{Ding2019} considered a Denniston $\{qt-q+t;t\}$-arc $D$ 
that is invariant under a cyclic group $G$ that fixes one point of $D$ and acts regularly on the set of the remaining 
points of $D$. Since $t=2^r$ for some $r$, condition $(t-1) |(q-1)$ means that $q$ is a power of $t$.  
Let $\mathbb{F}_t$ be the unique subfield of order $t$ of $F$.  
Choose  $\Delta$ as the additive subgroup of the field  $\mathbb{F}_t$. 
Cyclic group $G$ is given as $G=G_1 \times G_2$, where   
$$
G_1=\left\{ \left(
\begin{array}{cc}
 a+\delta b & b \\
 b & a   
\end{array}
\right) : a, b \in F, \ a^2+\delta ab +b^2 =1 \right\},  
$$
$$
G_2=\left\{ \left(
\begin{array}{cc}
 c & 0 \\
 0 & c  
\end{array}
\right) : c \in \mathbb{F}_t^* \right\}. 
$$
$G_1$ is a cyclic group of order $q+1$ and $G_2$ is a cyclic group of order $t-1$. Since $\gcd(q+1,t-1)=1$, $G$ is a 
cyclic group of order $(q+1)(t-1)$.

Now we consider this Denniston arc in terms of polar coordinates. 
Let $N=(q+1)(t-1)$,  and let $S'$ be the group of $N$-th roots of unity in $K$.  
If $\theta$ is a primitive element of the field $K$ then $S' =\langle \alpha \rangle$, where $\alpha = \theta^{(q-1)/(t-1)}$. 
We denote also $\beta = \theta^{q-1}$, thus $S =\langle \beta \rangle$. In polar coordinates we have 
$\alpha = \mu v$, where $\mu \in F$ and $v\in S$. 
Since $1 = \alpha^{(q+1)(t-1)} = (\mu v)^{(q+1)(t-1)} =\mu^{(q+1)(t-1)} = \mu^{2(t-1)}$, 
one has $ \mu^{t-1}=1$ and $\mu \in \mathbb{F}_t$. 
Therefore, $\mu$ is a primitive element  of $\mathbb{F}_t$ and 
$\langle \alpha \rangle = \langle \mu \rangle  \langle v \rangle$.  

Let $\Lambda =\mathbb{F}_t^+$ and define 
$$D = \bigcup_{\lambda \in \Lambda} \lambda S \subset K.$$ 
Then 
\begin{eqnarray*}
D 
&  = & \{ 0\} \cup S \cup  \alpha S \cup \dots \cup \alpha^{t-2} S \\
&  = &  \{0, 1, \alpha, \alpha^2, \dots, \alpha^{N-1} \} ,
\end{eqnarray*}
since $\mu^i S = \alpha^{i}S$ and $S=\langle \alpha^{t-1} \rangle$.   
This  maximal arc $D$ corresponds to the Denniston arc considered in \cite{Ding2019}. In this form we immediately 
see the cyclicity of $D$, and we call such an arc as {\it cyclic} Denniston maximal arc. 
Note that instead of $\alpha$ one can take any generator of $S'$.

Let $\alpha^i=x_i+y_i \xi$, $x_i,y_i \in F$. De Winter et al. \cite{Ding2019}  considered an extended cyclic code 
which, in fact, is equivalent to a code $C$ with generator matrix 
$$
A=\left(
\begin{array}{ccccc}
	x_0 & x_1 & \dots & x_{N-1}  & 0 \\
	y_0 & y_1 & \dots & y_{N-1}  & 0 \\
	1 & 1 & \dots & 1  & 1
\end{array}. 
\right) 
$$  
They showed that $C$ is a code with parameters $[N+1,3,N+1-t]$ and nonzero weights $N+1-t$ and $N+1$ 
(it can be also deduced using Lemma \ref{main}). Furthermore, the dual minimum distance $d^\perp$ of the code 
$C$ is $3$ when $t>2$ and $4$ when $t =2$ (hyperoval case). 

We consider now the reverse process. 

\begin{theorem}  
\label{max-arc}
An extended cyclic code over $\mathbb{F}_q$ with parameters $[qt-q+t,3,qt-q]$, $1< t <q$,  $q$ is a power of $t$, 
is equivalent to a code obtained from a cyclic Denniston maximal arc. 
\end{theorem} 

\begin{proof} 
First of all, we note that the case $t=2$ is a hyperoval case and it was considered in Theorem \ref{ext-MDS}, 
so we can assume that $t\ge 4$.  
Let $C$ be an  extended cyclic $[N+1,3,N+1-t]$-code. 
Let $C'$ denote the code obtained by puncturing the extended coordinate in $C$. Then $C'$ is a cyclic $[N,3]$-code.  

Recall that $S' =\langle \alpha \rangle$ and  $S =\langle \beta \rangle$, where $\alpha = \theta^{(q-1)/(t-1)}$, $\beta = \theta^{q-1}$ and  $\theta$ is a primitive element of the field $K$. Furthermore, $\alpha = \mu v$ for a primitive 
element $\mu$  of $\mathbb{F}_t$ and $\langle v\rangle = S$. 
Considering the set $S'$ with respect to the action of conjugation, we can write $S'= S_0 \cup S_1\cup S_2$, 
where $S_0 = \{u \in S' \mid u=\bar{u}\} = S' \cap F$, $S_1 \cap F = \varnothing$, $S_2= \{\bar{u} \mid u\in S_1\}$,  
$S_1 \cap S_2 = \varnothing$. Then 
$$x^N-1 =  \prod_{u\in S_0} (x-u) \prod_{u\in S_1} (x-u)(x-\bar{u}).$$
Note that $(x-u)(x-\bar{u}) = x^2-T(u)x+1\in F[x]$ is irreducible over $F$ for all $u\in S_1$.  
Consequently, every polynomial of degree 3 over $F$ dividing $x^N-1 $ must be of the form 
$(x-a)(x-b)(x-c)$ with distinct $a, b, c$ from $\mathbb{F}_t^*$, or $(x-\gamma)(x-\bar{\gamma})(x-a)$ with 
$\gamma \in S_1$, $a \in \mathbb{F}_t^*$.  We consider these two cases. 

1) Let $(C')^\perp$ be a cyclic $[N,N-3]$-code with generator polynomial $(x-a)(x-b)(x-c)$ with distinct $a, b, c$ 
from $\mathbb{F}_t^*$. 
Then by \cite[Theorem 4.4.3]{Huffman} we have that $(C')^\perp$ is a cyclic code with a parity-check matrix 
$$
B=\left(
\begin{array}{ccccc}
 1 & a & a^2 & \dots & a^{N-1}   \\
 1 & b & b^2 & \dots & b^{N-1}   \\
 1 & c & c^2 & \dots & c^{N-1}  
\end{array}
\right). 
$$
Hence $C'$ has generator matrix $B$. Since $C$ is an extended code of $C'$,  $C$ has a generator matrix of  
the following form: 
$$
A=\left(
\begin{array}{cccccc}
 1 & a & a^2 & \dots & a^{N-1}  & * \\
 1 & b & b^2 & \dots & b^{N-1}  & *  \\
 1 & c & c^2 & \dots & c^{N-1}  & * 
\end{array}
\right). 
$$     
Since $a^{t-1} =b^{t-1} =c^{t-1} =1$, every column $(a^i,b^i,c^i)^T$ in the matrix $A$ is repeated at least $q+1>t$ 
times, but it is not possible by Lemma  \ref{main}, since the minimum distance of $C$ is $N+1-t$.

2) Let $(C')^\perp$ be a cyclic $[N,N-3]$-code with generator polynomial $(x-\gamma)(x-\bar{\gamma})(x-a)$ 
with $\gamma \in S_1$, $a \in \mathbb{F}_t^*$. 
Define the following $2\times (q+1)$ matrix  over $K$: 
$$
L=\left(
\begin{array}{ccccc}
1 & \gamma & \gamma^2 & \dots & \gamma^{N-1}  \\
1 & a & a^2 & \dots & a^{N-1}
 \end{array}
\right).  
$$
Then by \cite[Theorem 4.4.3]{Huffman} we have 
$$ (C')^\perp = \{ c \in F^N \mid L c^T =0 \}. $$ 

 Let $\xi \in K \setminus F$ and $\gamma^i =x_i+y_i\xi$, where $x_i$ and  $y_i$ are from $F$. 
 Then $(C')^\perp$ is a cyclic code with a parity-check matrix 
$$
B=\left(
\begin{array}{cccc}
 x_0 & x_1 & \dots & x_{N-1}   \\
 y_0 & y_1 & \dots & y_{N-1}  \\
  1 & a & \dots & a^{N-1}  
\end{array}
\right). 
$$
This means that  $C'$ is a cyclic code over $F$ with generator matrix $B$. 

We claim that  $a = 1$. Indeed, 
assume that $a\ne 1$. Then $\sum_{i=0}^{t-2} a^i=0$, and consequently $\sum_{i=0}^{N-1} a^i=0$. In addition,   
$\sum_{i=0}^{N-1} \gamma^i = 0$. Hence the extended code  $C$ has the following generator matrix 
$$
A=\left(
\begin{array}{ccccc}
 x_0 & x_1 & \dots & x_{N-1}  & 0 \\
 y_0 & y_1 & \dots & y_{N-1}  & 0\\
  1 & a & \dots & a^{N-1}  & 0
\end{array}
\right).  
$$     

The nonzero columns of the matrix $A$ represent points $(x_i,y_i,a^i)$ in the projective plane $PG(2,q)$.  
Equivalently, in other representation, they are the points $(\gamma^i,a^i) \in PG(2,q)$, or  $((\gamma/a)^i,1)\in PG(2,q)$. 
Let $r$ be the smallest positive integer such that $(\gamma/a)^r=1$. 
Then in the sequence of the points $(\gamma/a)^i$, $0\le i\le N-1$, there are $r$ distinct points, repeated $N/r$ 
times. Therefore, the columns of the matrix $A$ represent distinct $r$ points in $PG(2,q)$,  repeated $N/r$ times, 
and one zero column more. Denote by $T$ the set of these distinct $r$ points in $PG(2,q)$. 

Assume that any line in $PG(2,q)$ intersects the set $T$ in maximum $s$ points. Then $T$ is an $\{r;s\}$-arc, 
and the standard bound for this arc gives
\begin{equation}
\label{arc}
r \le (q+1)(s-1) +1. 
\end{equation}

On the other hand, the codewords of the code $C$ can be written as $(a',b',c') A$, where $a',b',c' \in F$. 
Let a projective line $a'x+b'y+c'z=0$  intersect $T$ in $s$ points. Then the codeword $(a',b',c') A$ has 
$s\cdot \frac{N}{r} +1$ zero coordinates (the extra one comes from the zero column of $A$). 
Since the minimum distance of $C$ is $N+1-t$, we have 
$$s\cdot \frac{N}{r} +1 \le t,$$
which implies 
$$s\cdot \frac{N}{t-1} \le r,$$
$$s\cdot (q+1) \le r.$$
But the last inequality contradicts inequality (\ref{arc}).

Therefore, $a=1$ and $C$ has the following generator matrix 
$$
A=\left(
\begin{array}{ccccc}
 x_0 & x_1 & \dots & x_{N-1}  & 0 \\
 y_0 & y_1 & \dots & y_{N-1}  & 0\\
  1 & 1 & \dots & 1  & 1
\end{array}
\right). 
$$     

Now we will show that  the columns of the matrix $A$ are distinct, which means that the elements 
$\gamma^{i}$, $0\le i \le N-1$, are distinct. Indeed, assume that they are not distinct, and 
let $r$ be the smallest positive integer such that $\gamma^r=1$. Then $r<N$ and 
in the sequence of the elements $\gamma^i$, $0\le i\le N-1$, there are $r$ distinct elements, repeated $N/r$ 
times. Therefore, the columns of the matrix $A$ represent distinct $r$ points in $PG(2,q)$,  repeated $N/r \ge 2$ times, 
and one extra point $(0,0,1)^T$. Denote by $T$ the set of these distinct $r$ points in $PG(2,q)$. 

Assume that any line in $PG(2,q)$ intersects the set $T$ in maximum $s$ points. Then $T$ is an $\{r;s\}$-arc, 
and the standard bound for this arc gives
$$r \le (q+1)(s-1) +1. $$

On the other hand, the codewords of the code $C$ can be written as $(a',b',c') A$, where $a',b',c' \in F$. 
Let a projective line $a'x+b'y+c'z=0$  intersect $T$ in $s$ points. Then the codeword $(a',b',c') A$ has 
$s\cdot \frac{N}{r} $ or $s\cdot \frac{N}{r} +1$ zero coordinates. 
Since the minimum distance of $C$ is $N+1-t$, we have 
$$s\cdot \frac{N}{r}  \le t,$$
which implies 
$$2s \le t.$$
Futhermore, 
$$s\cdot \frac{N}{t}  \le r \le (q+1)(s-1) +1,$$
$$s\cdot \frac{(q+1)(t-1)}{t}   \le (q+1)(s-1) +1,$$
$$s (q+1)(1- \frac{1}{t})   \le (q+1)s-(q+1) +1,$$
$$-s (q+1)\frac{1}{t}   \le -q,$$
$$s(q+1) \ge qt \ge 2qs,$$
a contradiction. 

Therefore, the elements $\gamma^{i}$, $0\le i \le N-1$, are distinct, the order of $\gamma$ is equal to $N$, 
$\gamma$ generates $S'$ and the set 
$\{0, 1, \gamma^{1}, \cdots, \gamma^{N-1}\}$ is a cyclic Denniston maximal arc. \qed 
\end{proof}

\section{Cyclic codes and ovoids} 
\label{sec-ovoids} 

In this section we show that the results from previous sections can be extended to ovoids.  

In $PG(n,q)$, $n\ge 3$,  a set $\mathcal K$ of $k$ points no three of which are collinear is called a $k$-{\it cap}. 
For any $k$-cap $\mathcal K$ in $PG(3,q)$ with $q\ne 2$, the cardinality $k$ satifies $k\le q^2+1$. 
A $(q^2+1)$-cap of  $PG(3,q)$, $q\ne 2$, is called an {\it ovoid}.  Any plane of $PG(3,q)$ meets an ovoid in 
$1$ or $q+1$ points \cite{Hir}. 

A linear $[q^2+1,4]$-code is called an {\it ovoid code} if the columns of its generator matrix $G$ constitute 
an ovoid in $PG(3,q)$.

Let $V$ be a finite dimensional vector space over a field $F$ of characteristic 2. 
A {\it quadratic form} on $V$ is a mapping $Q: V \rightarrow F$ such that  
 
\begin{enumerate} 
\item 
$Q(\lambda x) = \lambda^2 Q(x)$ for all $\lambda \in F$, $x\in V$, and 
\item 
$B(x,y) = Q(x+y)+Q(x)+Q(y)$ is a bilinear from.  
\end{enumerate} 
A quadratic form is {\it non-degenerate} if the property $B(x,y)=Q(x)=0$ for all $y\in V$ implies $x=0$. 
A vector $x\in V$ is {\it singular} if $Q(x)=0$. 
The set of singular points of $Q$ defines a {\it quadric} in the projective space $P(V)$ 
($P(V)$ is a projective space obtained from $V\setminus \{ 0\}$ by identifying nonzero vectors 
$v\in V$ with $\lambda v$ for all $\lambda \in F$). 
 
Let $Q$ be a non-degenerate quadratic form on $4$-dimensional vector space $V$ over $F$. 
The coordinate frame can be chosen so that $Q$ is equivalent to one of the following two expressions: 
\begin{enumerate} 
\item 
$x_0x_1+x_2x_3$, or 
\item 
$x_0^2+ax_0x_1+x_1^2+x_2x_3$, where $a\in F$ and the polynomial $\xi^2+a\xi+1$ is irreducible over $F$.   
\end{enumerate} 
In the former case the quadratic form $Q$ defines a {\it hyperbolic} quadric in $PG(3,q)$, 
and in the latter case $Q$ defines an {\it elliptic} quadric.  
The hyperbolic quadric in $PG(3,q)$ contains $(q+1)^2$ points, and the elliptic quadric in $PG(3,q)$ 
contains $q^2+1$ points \cite{Bierbrauer}. The elliptic quadric in $PG(3,q)$ is an ovoid. 
 

The next theorem provides a coordinate-free presentation of the elliptic quadric in $PG(3,q)$. 
 
 \begin{theorem}  
 \label{quadric}
Let $E \supset K \supset F$ be a chain of finite fields, $|E|=q^4$, $|K|=q^2$, $|F|=q$, $q=2^m$.  Then 
$$Q(x)= Tr_{K/F}(N_{E/K}(x))$$ 
is a non-degenerate quadratic form on $4$-dimensional vector space  $E$ over $F$. 
Moreover,  the set 
$$\mathcal{O}=\{u\in E \mid N_{E/K}(u)=1\}=\{u\in E \mid u^{q^2+1}=1\}$$ 
determines an elliptic quadric in $PG(3,q)$. 
\end{theorem} 

\begin{proof} 
We have 
$$N_{E/K}(x) = x\cdot x^{q^2} = x^{q^2+1},$$ 
$$Q(\lambda x) = Tr_{K/F}(N_{E/K}(\lambda x)) =\lambda^2 Tr_{K/F}(N_{E/K} (x)) =
\lambda^2 Q(x) \ {\rm for \ all} \  \lambda\in F,$$
$$N_{E/K}(x+y) = (x+y) (x+y)^{q^2} = xx^{q^2} + xy^{q^2} +x^{q^2}y +yy^{q^2},$$ 
\begin{eqnarray*}
Q(x+y)+Q(x)+Q(y) &=& Tr_{K/F}(N_{E/K}(x+y)+N_{E/K}(x)+N_{E/K}(y)) \\
                              &=& Tr_{K/F}(xy^{q^2} +x^{q^2}y).
\end{eqnarray*}
It is clear that $B(x,y) = Tr_{K/F}(xy^{q^2} +x^{q^2}y)$ is a bilinear non-degenerate form. 

Now we show that 
$$\{x\in E \mid Q(x)=0\} = \{\lambda u \in E \mid \lambda\in F, u^{q^2+1}=1\}.$$ 
Since $[K:F]=2$, $Q(x)=Tr_{K/F}(N_{E/K}(x)) =0$ if and only if $N_{E/K}(x)\in F$. 
Further, any element $x\in E$ can be written as $x=\lambda u$, 
where $\lambda \in K$, $u\in E$, $u^{q^2+1} =1$ (polar presentation of elements of $E$ over $K$). 
Then $\lambda^2 = N_{E/K}(x) \in F$. Hence $\lambda \in F$.   

Therefore, up to multiplication by elements from $F$, all solutions of $Q(x)=0$ are elements of $\mathcal{O}$. 
Since $|\mathcal{O}|=q^2+1$, $\mathcal{O}$ determines an elliptic quadric in $PG(3,q)$.  \qed 
\end{proof}
 
Note that a more general version of the previous quadratic form is given in \cite{Shult}, but in our particular case 
our presentation looks simpler  and it is very suitable for our next investigations. 
 
Ding \cite[Section 13.3]{Ding} introduced a family of cyclic codes with parameters $[q^2+1,4,q^2-q]$ and stated 
without proof that they can be  obtained from elliptic quadrics. The next theorem proves this statement and 
shows a very natural connection between  these cyclic codes and elliptic quadrics. 
 
 \begin{theorem}  
\label{ovoid}
A cyclic code over $\mathbb{F}_q$ with parameters $[q^2+1,4,q^2-q]$    
is equivalent to an ovoid code obtained from an elliptic quadric in $PG(3,q)$. 
\end{theorem} 

\begin{proof} 
Let $E \supset K \supset F$ be a chain of finite fields, $|E|=q^4$, $|K|=q^2$, $|F|=q$, $q=2^m$. 
Consider the set 
$$\mathcal{O} =\{u\in E \mid N_{E/K}(u)=1\}= \{u\in E \mid u^{q^2+1}=1\}.$$ 
The factorization of $x^{q^2+1}-1$ over $E$ is given by 
$$x^{q^2+1}-1 = \prod_{u\in \mathcal{O}} (x-u)= (x-1) \prod_{u\in \mathcal{O}, u\ne 1} (x-u).$$ 
Let $f(x) \in F[x]$ be an irreducible divisor of $x^{q^2+1}-1$ which has a root $\gamma \in  \mathcal{O}$, 
$\gamma \ne 1$. Then $\sigma(\gamma)$ is also a root of $f(x)$ for all $\sigma \in Gal(E/F)$. 
Therefore,  $\gamma^q$, $\gamma^{q^2}$ and $\gamma^{q^3}$ are roots of $f(x)$. 
We claim that elements $\gamma$, $\gamma^q$, $\gamma^{q^2}$, $\gamma^{q^3}$ are distinct. 
Indeed, if they are not distinct, then $\gamma =\gamma^{q^2}$. 
Therefore, $\gamma \in K$ and $\gamma$ belongs to the unit circle in polar presentation of $E$ over $K$. 
Hence $\gamma =1$, a contradiction. 
Consequently, every polynomial of degree 4 over $F$ dividing $x^{q+1}-1 $ must be of the form 
$(x-\gamma)(x-\gamma^q)(x-\gamma^{q^2})(x-\gamma^{q^3})$ with $\gamma \in \mathcal{O} \setminus \{1\}$.  

Let $C$ be a cyclic $[q^2+1,4,q^2-q]$-code. Then the parity-check polynomial of $C$ is 
$(x-\gamma)(x-\gamma^q)(x-\gamma^{q^2})(x-\gamma^{q^3})$ for some  $\gamma \in \mathcal{O} \setminus \{1\}$.
Then $C^\perp$ is a cyclic $[q+1,q-2]$-code with generator polynomial 
$(x-\gamma^{-1})(x-\gamma^{-q})(x-\gamma^{-q^2})(x-\gamma^{-q^3})$. 
Define the $1\times (q^2+1)$ matrix  
$$
L=\left(
\begin{array}{ccccc}
1 & \gamma^{-1} & \gamma^{-2} & \dots & \gamma^{-q^2}  
 \end{array}
\right).  
$$
Then by \cite[Theorem 4.4.3]{Huffman} we have 
$$ C^\perp = \{ c \in F^{q^2+1} \mid L c^T =0 \}. $$ 

Let $\{ 1, \xi\, \zeta, \eta \}$ be a basis of $K$ over $F$. 
Let $\gamma^{-i} =x_i+y_i\xi +z_i\zeta + w_i\eta$, where $x_i, y_i , z_i, w_i\in F$. 
Then $C^\perp$ is a cyclic code with the parity-check matrix 
$$
A=\left(
\begin{array}{cccc}
 x_0 & x_1 & \dots & x_{q^2}   \\
 y_0 & y_1 & \dots & y_{q^2}   \\
 z_0 & z_1 & \dots & z_{q^2}   \\
  w_0 & w_1 & \dots & w_{q^2}   
\end{array}
\right). 
$$
Hence $C$ is a cyclic code with  generator matrix $A$. 

Any linear code over $\mathbb{F}_q$ with parameters $[q^2+1,4,q^2-q]$  is a projective two-weight  code 
(\cite[p. 192]{Bierbrauer} and \cite{Ding2018}).  Hence, $C$ is an ovoid code \cite{Calderbank}. 
Therefore, the columns of the matrix $A$ are distinct, and $\gamma$ is a multiplicative generator of $\mathcal{O}$. 
Then Theorem \ref{quadric} implies  that $C$ is an ovoid code obtained from the elliptic quadric $\mathcal{O}$. 
\qed 
\end{proof}

\bigskip

{\bf Acknowledgments}

\medskip

The authors would like to thank Cunsheng Ding for valuable discussions and suggestions. 
The author is also grateful to the anonymous reviewers for their detailed comments that improved 
the presentation and quality of this paper.
This work was supported by UAEU grant 31S366.

%
%



\end{document}